\documentclass[12pt]{amsart}
\usepackage{latexsym}
\usepackage{hyperref}

\usepackage{amsmath}

\newcommand{\ignore}[1]{}

\newcommand{\hide}[1]{}

\DeclareMathOperator{\ad}{ad}

\newcommand{\F}{\mathbb F}

\newcommand{\N}{\mathbb N}

\newtheorem{dummy}{Dummy}

\newtheorem{lemma}[dummy]{Lemma}

\newtheorem{theorem}[dummy]{Theorem}

\newtheorem{prop}[dummy]{Proposition}
\newtheorem{cor}[dummy]{Corollary}

\theoremstyle{definition}
\newtheorem{definition}[dummy]{Definition}

\newtheorem{convention}[dummy]{Convention}

\theoremstyle{remark}
\newtheorem{rem}[dummy]{Remark}
\newtheorem*{rem*}{Remark to ourselves}

\begin{document}

\bibliographystyle{amsalpha}

\author{Marina Avitabile}
\email{marina.avitabile@unimib.it}
\address{Dipartimento di Matematica e Applicazioni\\
  Universit\`a degli Studi di Milano - Bicocca\\
 via Cozzi 55\\
  I-20125 Milano\\
  Italy}
\author{Sandro Mattarei}
\email{smattarei@lincoln.ac.uk}
\address{Charlotte Scott Centre for Algebra\\
University of Lincoln \\
Brayford Pool
Lincoln, LN6 7TS\\
United Kingdom}

\title[Diamond distances in Nottingham algebras]{Diamond distances in Nottingham algebras}

\subjclass[2020]{Primary 17B50; secondary 17B70, 17B65}
\keywords{Modular Lie algebra, graded Lie algebra, thin Lie algebra}

\begin{abstract}
{\em Nottingham algebras} are a class of just-infinite-dimensional, modular,
$\N$-graded Lie algebras,
which includes the graded Lie algebra associated to the Nottingham group with respect to its lower central series.
Homogeneous components of a Nottingham algebra have dimension one or two, and in the latter case they are called {\em diamonds}.
The first diamond occurs in degree $1$, and the second occurs in degree $q$, a power of the characteristic.
Many examples of Nottingham algebras are known,
in which each diamond past the first can be assigned a {\em type,}
either belonging to the underlying field or equal to $\infty$.

A prospective classification of Nottingham algebras requires describing all possible diamond patterns.
In this paper we establish some crucial contributions towards that goal.
One is showing that all diamonds, past the first,
of an {\em arbitrary} Nottingham algebra $L$
can be assigned  a type,
in such a way that the degrees
and types of the diamonds completely describe $L$.
At the same time we prove that
the difference in degrees of any two consecutive diamonds in any Nottingham algebra equals $q-1$.
As a side-product of our investigation, we classify the Nottingham algebras where all diamonds have type $\infty$.
\end{abstract}

\maketitle
\section{Introduction}\label{sec:intro}

A {\em thin} Lie algebra is a graded Lie algebra $L=\bigoplus_{i=1}^{\infty}L_i$
with $\dim L_1=2$ and satisfying the following {\em covering property:}
for each $i$, each nonzero $z\in L_i$ one has $[zL_1]=L_{i+1}$.
(Note that we write Lie products without a comma.)
This implies at once that homogeneous components of a thin Lie algebra are at most two-dimensional.
Those components of dimension two are called {\em diamonds,}
hence $L_1$ is a diamond, and if there are no other diamonds then $L$ is a {\em graded Lie algebra of maximal class}~\cite{CMN,CN}.
It is convenient, however, to explicitly exclude graded Lie algebras of maximal class from the definition of thin Lie algebras.
Thus, a thin Lie algebra must have at least one further diamond besides $L_1$ (which we may call the {\em first} diamond of $L$),
and we let $L_k$ be the earliest such diamond (the {\em second} diamond).
In this paper we conveniently assume all thin Lie algebras to have infinite dimension.
We briefly comment on the finite-dimensional case in Remark~\ref{rem:finite-dim}.

The term {\em diamond} originates from a lattice-theoretic characterization of thin Lie algebras motivated by~\cite{Br}.
In fact, any nonzero graded ideal $I$ of a thin Lie algebra $L$ is constrained
between two consecutive Lie powers of $L$, in the sense that $L^i\supseteq I\supseteq L^{i+1}$ for some $i$,
and hence the lattice of graded ideals looks like a sequence of diamonds
(a name for the lattice of subspaces of a two-dimensional vector space) connected by {\em chains}.
We will not formally assign numerical lengths to those chains but it should be clear that they are important in describing the structure of a thin Lie algebra.
Knowing those lengths amounts to knowing the degrees in which the diamonds occur.

The most basic invariant of a thin Lie algebra is the degree $k$ of the second diamond.
It is very easy to see that $k$ must be odd.
It is known from~\cite{CMNS} and~\cite{AviJur}
(but see also~\cite{Mat:chain_lengths} for a revised treatment of a portion of the argument)
that if the characteristic $p$ is different from $2$ then $k$ can only be one of $3$, $5$, $q$, or $2q-1$, where $q$ is a power of $p$.
In particular, only $3$ and $5$ can occur in characteristic zero.
In fact, in characteristic zero or larger than $5$, thin Lie algebras with $k=3$ or $5$,
with the further assumption $\dim(L_4)=1$ in the former case,
were shown in~\cite{CMNS} to belong to up to three isomorphism types,
associated to $p$-adic Lie groups of types $A_1$ and $A_2$. (Explicit realizations can be found in~\cite{Mat:thin-groups}.)
In contrast, the values $q$ and $2q-1$ for $k$ occur for two wide classes of thin Lie algebras built from certain nonclassical finite-dimensional
simple modular Lie algebras, and also for thin Lie algebras obtained from graded Lie algebras of maximal class through various constructions.

In this paper we focus on thin Lie algebras with second diamond $L_q$.
One remarkable example of such thin Lie algebras arises as the graded Lie algebra associated to the lower central series
of the {\em Nottingham group} over the prime field $\F_p$, for $p$ odd~\cite{Car:Nottingham}.
That algebra has its second diamond in degree $p$, but admits a natural generalization with a power $q$ of $p$ in place of $p$.
For this reason thin Lie algebras with second diamond $L_q$ have been called {\em Nottingham algebras} in the literature.
However, because of exceptional behaviour in small characteristics here we reserve the name {\em Nottingham algebras} to thin Lie algebras
of characteristic $p>3$, having second diamond $L_q$ with $q>5$.

A wide variety of Nottingham algebras are known.
Several arise from certain cyclic gradings of various simple Lie algebras of Cartan type.
In particular, the thin Lie algebra associated with the Nottingham group arises from a cyclic grading of the {\em Witt algebra}.
Further Nottingham algebras, and in fact uncountably many ones, are closely related to graded Lie algebras of maximal class.
We refer the reader to Theorem~\ref{thm:k=q} and the discussion which follows it for a comprehensive survey.

In all Nottingham algebras which have been studied so far,
each diamond past the first can be attached a {\em type},
which is an element of the underlying field or possibly $\infty$.
The second diamond $L_q$ has invariably type $-1$, and we assign no type to the first diamond $L_1$.
The type of a diamond $L_m$ describes the adjoint action of $L_1$ on $L_m$, in such a way that knowledge of
all degrees in which diamonds occur in $L$, and their types, determines $L$ up to isomorphism.
It is necessary to include {\em fake diamonds} in such a description.
Those are in fact one-dimensional components, as we explain in Section~\ref{sec:types}, and correspond to types $0$ and $1$.
We refer to diamonds which are not fake, and thus are two-dimensional, as {\em genuine} diamonds.

It is by no means obvious that any two-dimensional component
of an arbitrary Nottingham algebra can be assigned a type.
Proving that this is indeed the case constitutes
a substantial step in an ambitious
classification project of Nottingham algebras.
Another step is proving that any two consecutive diamonds
occur at a degree difference of $q-1$.
We achieve both of those goals in this paper.
The statement on the degree difference needs proper interpretation in case of fake diamonds, as we explain below.

\begin{theorem}\label{thm:distance_intro}
Let $L$ be a Nottingham algebra, with second diamond $L_q$,
and let $y$ be a nonzero element of $L_1$ which centralizes $L_2$.
The following assertions hold.
\begin{itemize}
\item[(a)]
No more than two consecutive homogeneous components
of $L$ can be not centralized by $y$.
\item[(b)]
If $[L_{m-1}y]\neq 0$
and $[L_{m}y]\neq 0$,
then $L_m$ is two-dimensional (a genuine diamond) and can be
assigned a type.
\item[(c)]
If $[L_{m-2}y]=0$,
$[L_{m-1}y]\neq 0$
and $[L_{m}y]=0$,
then $L_{m}$ is a fake diamond of type $0$.
Equivalently, $L_{m-1}$ may be viewed as a fake diamond of type $1$.
\item[(d)]
The difference in degree of any two consecutive
(possibly fake) diamonds equals $q-1$,
for an appropriate interpretation of the fake diamonds if present.
\end{itemize}
\end{theorem}

Some explanations are in order on the statement
of Theorem~\ref{thm:distance_intro}, which we briefly sketch here,
referring to Sections~\ref{sec:types} and~\ref{sec:chains} for details.
According to Theorem~\ref{thm:distance_intro}, if all diamonds of $L$
are genuine, that is, two-dimensional, they must occur in all degrees
congruent to $1$ modulo $q-1$.
However, in the presence of fake diamonds
it is not possible to draw that conclusion,
which is indeed untrue in general.
In fact, whenever a one-dimensional
component $L_m$ of a Nottingham algebra
is not centralized by $y$, we either call $L_m$ a diamond of type $1$,
or $L_{m+1}$ a diamond of type $0$.
That is because both situations fit in as special cases of a generic definition of diamond type, see Definitions~\ref{def:type}
and~\ref{def:type-fake}.
We never simultaneously
call both $L_m$ and $L_{m+1}$ (fake) diamonds,
but make a choice of precisely
one of them, see Convention~\ref{convention:fake}.
There is generally no intrinsic reason for this choice,
but we take advantage of this ambiguity when calculating the degree difference between consecutive diamonds in case at least
one of them is fake
(and the choice might be different for a given fake diamond
when computing its distance from the previous diamond,
or the next one).
We prove Theorem~\ref{thm:distance_intro} in Section~\ref{sec:chains},
as immediate consequence of a more precise result,
Theorem~\ref{thm:distance}.
Both depend on intermediate, more technical results, proved in
Sections~\ref{sec:chains} and~\ref{sec:chain_proof}.

Various patterns of diamond types occur in Nottingham algebras.
One possible pattern has all diamonds of infinite type, with the necessary exception of the first two.
Such algebras, which were constructed in~\cite{Young:thesis},
have second diamond $L_q$, of type $-1$,
and diamonds of infinite type in each higher degree
congruent to $1$ modulo $q-1$.
Now the locations of those diamonds are a consequence of
Theorem~\ref{thm:distance_intro}, and we obtain at once
the following simple instance of a classification result.

\begin{theorem}\label{thm:all_infty}
There exists a unique Nottingham algebra
with second diamond in degree $q$ and all other diamonds having infinite type.
\end{theorem}

\section{Nottingham algebras}\label{sec:types}

Recall from the Introduction that a {\em thin} Lie algebra is a graded Lie algebra $L=\bigoplus_{i=1}^{\infty}L_i$,
with $\dim L_1=2$ and satisfying the {\em covering property:}
for each $i$, each nonzero $z\in L_i$ satisfies $[zL_1]=L_{i+1}$.
However, in this paper we are assume $L$ to have infinite
dimension (but see Remark~\ref{rem:finite-dim}),
and it follows that $L$ has trivial centre.
As we mentioned in the Introduction, our definition of a Nottingham algebra includes restrictions on $p$ and $q$,
which we will partly justify below.

\begin{definition}
In this paper a {\em Nottingham algebra} is a thin Lie algebra, over a field of characteristic $p>3$,
with second diamond $L_q$, where $q>5$ is a power of $p$.
\end{definition}

In this paper we use the left-normed convention
for iterated Lie products, hence $[abc]$ stands for $[[ab]c]$.
We also use the shorthand
$[ab^i]=[ab\cdots b]$,
where $b$ occurs $i$ times.

Let us start a bit more generally and consider first a thin Lie algebra $L$ with $\dim(L_3)=1$.
Then there is a nonzero element $y$ of $L_1$, unique up to a scalar multiple, such that $[L_2y]=0$.
Extending to a basis $x,y$ of $L_1$, this means $[yxy]=0$.
It is not hard to deduce from this relation that no two consecutive components in such a thin Lie algebra $L$
can both be diamonds, see~\cite{Mat:sandwich} for a proof.
Thus, any two diamonds are separated by one or more one-dimensional homogeneous components.

More delicate arguments in~\cite{Mat:sandwich}
show that any thin Lie algebra $L$
with second diamond occurring past $L_5$
satisfies $[Lyy]=0$, which means $(\ad y)^2=0$.
According to a well-known definition of Kostrikin's,
in odd characteristic that means $y$
is a {\em sandwich element} of $L$.
The significance of this fact is discussed in~\cite{Mat:sandwich}.
In particular, Nottingham algebras as defined here satisfy $(\ad y)^2=0$.
One of our reasons for excluding $q=5$ from our definition of a Nottingham algebra
is the existence of the thin Lie algebra of classical type $A_2$ of~\cite{CMNS} (and~\cite{Mat:thin-groups})
which has a diamond in each degree congruent to $\pm 1$ modulo $6$ (in any characteristic except $2$).
In particular, that thin Lie algebra has $L_5$ and $L_7$ as second and third diamond, which easily implies $[L_5yy]\neq 0$.

From now on let $L$ be a Nottingham algebra with second diamond $L_q$.
Then the element $y$ centralizes each homogeneous component from $L_{2}$ up to $L_{q-2}$.
That is an nontrivial assertion proved in~\cite{CaJu:quotients}, and relies on the theory of graded Lie algebras of maximal class established in~\cite{CMN,CN}.
Consequently, $L_i$ is spanned by $[yx^{i-1}]$ for $2\le i<q$.
In particular, $v_1=[yx^{q-2}]$ spans the component $L_{q-1}$ and, in turn, $[v_1x]$ and $[v_1y]$ span the second diamond $L_q$.
The subscript in $v_1$ is motivated by a convenience
of numbering the diamonds in a certain way,
and this notation turns out to be more natural than denoting it $v_2$ in some earlier papers.
However, such explicit numbering will not play any role in this paper.

It is now easy to see that one may redefine $x$ in such a way that
\[
[v_1xx]=0=[v_1yy] \quad\textrm{and} \quad [v_1yx]=-2[v_1xy],
\]
see~\cite[Section~3]{AviMat:A-Z} for a cleaner excerpt of the original argument in~\cite{Car:Nottingham}.
In the rest of this paper we refer to such $x$ and $y$ as {\em standard generators} of $L$.
Each of them is only determined up to a scalar multiple, but a different choice will not affect our definitions below,
in particular the definition of a diamond's type.
Because $[yx^q]=[v_1xx]=0$, we have $(\ad x)^q=0$.
Indeed, since $(\ad x)^q$ is a derivation of $L$, its kernel is a subalgebra, but then that must equal $L$ as both generators $x$ and $y$ belong to it.

We recall the definition of {\em type} of a diamond as introduced in ~\cite{CaMa:Nottingham}.
(Note that diamond types are defined differently for thin Lie algebras with second diamond $L_{2q-1}$, see~\cite{CaMa:thin}.)
We do not assign a type to the first diamond $L_1$.
Let then $L_m$ be a diamond past $L_1$, that is, a two-dimensional homogeneous component of $L$ with $m>1$.
Because no two consecutive homogeneous components can be diamonds, $L_{m-1}$ is one-dimensional, and so is $L_{m+1}$.
If $w$ spans $L_{m-1}$, then $L_m$ is spanned by $[wx]$ and $[wy]$,
and $L_{m+1}$ is spanned by $[wxx]$, $[wxy]$, $[wyx]$ and $[wyy]$.
The following definition encodes particular relations
among the latter four elements.

\begin{definition}\label{def:type}
Let $L$ be a Nottingham algebra,
with second diamond $L_q$ and standard generators $x$ and $y$.
Let $L_m$ be a diamond of $L$, with $m>1$,
and let $w$ be a nonzero element in $L_{m-1}$.
\begin{itemize}
\item[(a)]
We say $L_m$ is a diamond of {\em finite type} $\mu$, where $\mu\in\F$,
if
\begin{equation*}
[wxx]=0=[wyy] \quad\textrm{and} \quad \mu[wyx]=(1-\mu)[wxy].
\end{equation*}
\item[(b)]
We say $L_m$ is a diamond of {\em infinite type} if
\begin{equation*}
[wxx]=0=[wyy] \quad\textrm{and} \quad [wyx]=-[wxy].
\end{equation*}
\end{itemize}
\end{definition}

In particular, this definition applies to the second diamond $L_{q}$, which therefore has invariably type $\mu=-1$.

It is by no means obvious that a two-dimensional
component $L_m$ of a Nottingham algebra $L$ should satisfy
Definition~\ref{def:type} for some value of $\mu$.
More precisely, while $[wyy]=0$ follows from $(\ad y)^2=0$,
whence $[wyx]$ spans $L_{m+1}$ according to the covering property, and then
$\mu[wyx]=(1-\mu)[wxy]$ can always be attained by an appropriate choice of $\mu$,
the remaining relation $[wxx]=0$ is not easy to establish.
We will call a diamond $L_m$ {\em a diamond with a type} if $[wxx]=0$, whence those relations hold for some value of $\mu$ which we wish to leave unspecified.
This terminology will include the special cases $\mu=0$ and $\mu=1$, which we discuss next.
One of the consequences of our main result, Theorem~\ref{thm:distance_intro}
is that every two-dimensional component past $L_1$
of a Nottingham algebra satisfies Definition~\ref{def:type},
and hence is a diamond with a type.
(This conclusion makes our definition of diamond with a type
eventually redundant.)

The values $\mu=0$ and $\mu=1$ cannot actually occur in
Definition~\ref{def:type}.
If $\mu=0$ then the relations $[wxx]=0=[wxy]$
would imply that the element $[wx]$ is central,
and hence vanish because of blanket infinite-dimensionality
assumption on $L$.
Similarly, if $\mu=1$ then the element $[wy]$ would be central,
and hence vanish.
Thus, strictly speaking, diamonds of type $0$ or $1$ cannot occur,
at least if we insist that a diamond should have dimension two,
as in Definition~\ref{def:type}.
Nevertheless, it is convenient for a uniform description
of the diamonds patterns
in Nottingham algebras to allow ourselves to
call {\em diamonds of type $0$ or $1$}
certain one-dimensional homogeneous components $L_m$,
as long as they satisfy the relations of
Definition~\ref{def:type}
with $\mu=0$ or $1$.
This leads us to the following definition.

\begin{definition}\label{def:type-fake}
Let $L$ be a Nottingham algebra,
with second diamond $L_q$ and standard generators $x$ and $y$.
Let $L_{m-1}$ be a one-dimensional component, spanned by $w$, with $m>1$.
\begin{itemize}
\item[(a)]
We say $L_m$ is a diamond of {\em of type $1$} if
\begin{equation*}
[wxx]=0 \quad\textrm{and} \quad [wy]=0.
\end{equation*}
\item[(b)]
We say $L_m$ is a diamond {\em of type $0$} if
\begin{equation*}
[wx]=0 \quad\textrm{and} \quad [wyy]=0.
\end{equation*}
\end{itemize}
\end{definition}


We refer to diamonds of type $0$ or $1$ as {\em fake diamonds}
to distinguish them from the {\em genuine diamonds} of dimension two.
The necessity of including fake diamonds in a treatment
of Nottingham algebras
arises from the fact that in various notable instances diamonds occur at regular intervals, with types following an arithmetic progression
(see Theorem~\ref{thm:k=q} below).
When such arithmetic progression of types
passes through $0$ or $1$, fake diamonds occur.

However, this carries an inherent ambiguity:
whenever $L_m$ satisfies the definition of a diamond of type $1$
(which amounts to $[L_{m-1}y]=0$ and $[L_mx]=0$), the next homogeneous component $L_{m+1}$ satisfies the definition of a diamond of type $0$
(because $[L_my]=L_{m+1}$ due to the covering property, and then $[L_{m+1}y]=0$ due to $[Lyy]=0$).
Thus, if $w$ spans $L_{m-1}$ then $[wx]$ spans $L_m$
and $[wxy]$ spans $L_{m+1}$, and we have the relations
\begin{equation}\label{eq:fake_diamonds}
[wy]=0, \qquad
[wxx]=0, \qquad
[wxyy]=0.
\end{equation}
The first and second are those in part~(a) of
Definition~\ref{def:type-fake}, and the second and third
are those in part $(b)$ if we use $w'=[wx]$ instead of $w$ in it.
For various reasons it is inconvenient to simultaneously regard
two consecutive components as fake diamonds, and so we adopt the following convention.

\begin{convention}\label{convention:fake}
Whenever we have a diamond $L_m$ of type $1$, necessarily followed by a diamond $L_{m+1}$ of type $0$,
we allow ourselves to call (fake) diamond precisely one of $L_m$
and $L_{m+1}$, of the appropriate type, and not the other.
\end{convention}

In several cases there is a natural choice between
calling $L_m$ a diamond of type $1$, or $L_{m+1}$ a diamond of type $0$,
which makes diamonds (including the fake ones)
occur at regular distances,
with a difference of $q-1$ in degrees.
We illustrate that through the following existence result, which will be clarified and expanded in commentaries to follow.

\begin{theorem}\label{thm:k=q}
There exist Nottingham algebras $L$ with second diamond $L_q$,
where (possibly fake) diamonds occur in  each degree congruent to $1$ modulo $q-1$, and have types described by any of the following patterns:
\begin{itemize}
\item[(a)]
all diamonds of type $-1$~\cite{Car:Nottingham};
\item[(b)]
all diamonds of finite types following a non-constant arithmetic progression~\cite{Avi,AviMat:A-Z};
\item[(c)]
all diamonds of infinite type except for those in degrees $\equiv q\pmod{p^s(q-1)}$ for some $s>0$, which have type $-1$~\cite{AviMat:A-Z};
\item[(d)]
all diamonds of infinite type except for those in degrees $\equiv q\pmod{p^s(q-1)}$ for some $s>0$, which have finite types following a non-constant arithmetic progression~\cite{AviMat:mixed_types}.
\end{itemize}
\end{theorem}

Nottingham algebras as in case~(a) of Theorem~\ref{thm:k=q}
were explicitly constructed in~\cite{Car:Nottingham},
using a certain cyclic grading of Zassenhaus algebras.
The special case where $q=p$ is the graded Lie algebra associated with the lower central series of the Nottingham group, thus justifying their name.
The convenience of introducing fake diamonds was first observed in~\cite{CaMa:Nottingham},
where various Nottingham algebras were identified
through finite presentations,
with diamonds in all degrees congruent
to $1$ modulo $q-1$,
and of types following a non-constant arithmetic progression, as in case~(b) above.
If that arithmetic progression passes through $0$, that is, if it runs through the prime field (as the second diamond has invariably type $-1$),
then those diamonds must include fake diamonds, of both types $0$ and $1$.


In all cases of Theorem~\ref{thm:k=q}, each homogeneous component which is not a diamond or immediately precedes a diamond is centralized by $y$.
In fact, according to our Theorem~\ref{thm:distance_intro}
this is a general feature of Nottingham algebras,
provided that we appropriately include fake diamonds
when needed.
Consequently, specifying all degrees
in which diamonds occur (possibly fake and making use of Convention~\ref{convention:fake}),
and their types, describes a Nottingham algebra completely.
Because each of the Nottingham algebras of
Theorem~\ref{thm:k=q} has diamonds in each degree congruent to $1$ modulo $q-1$,
the distance between consecutive diamonds
is invariably $q-1$, provided that we assign
an appropriate type $0$ or $1$ to each fake diamond.

It is known that the algebras of Theorem~\ref{thm:k=q} (which over a finite field, say, are countably many),
constitute a minority among all Nottingham algebras.
This is due to two distinct constructions presented by David Young in his PhD thesis~\cite{Young:thesis},
where two Nottingham algebras $\mathcal{T}_{q,1}(M)$ and $\mathcal{T}_{q,2}(M)$, both with second diamond $L_q$,
are produced from any given graded Lie algebra $M$ of maximal class having at most two distinct two-step centralizers.
Because there are uncountably many such $M$ over any given field, and the Nottingham algebras thus produced are pairwise non-isomorphic,
it follows that there are uncountably many Nottingham algebras over any given field.
Young described further Nottingham algebras with second diamond $L_q$, namely a countable family of {\em Nottingham deflations} $N(q,r)$, where $r$ is a power of $p$,
and two Nottingham algebras $L_1(q)$ and $L_2(q)$ which have no further genuine diamond past $L_q$ (and hence have coclass two).
One may also regard $L_1(q)$ as a limit case $N(q,\infty)$.

In each of the Nottingham algebras constructed in~\cite{Young:thesis}, all diamonds past $L_q$ have infinite type or are fake.
We refrain from providing further details on those algebras algebras except for $L_1(q)$ and $L_2(q)$,
which are easy to describe and illustrate an important point on diamond distances.
Beyond their second diamond $L_q$, each of them has a fake diamond in each degree $m>q$ with $m\equiv -1\pmod{q}$,
but those fake diamonds are all of type $1$ in case of $L_1(q)$, and all of type $0$ in case of $L_2(q)$.
Thus, in both algebras we observe distances of $q$ rather than $q-1$ between certain fake consecutive diamonds of the same type (either both $0$ or both $1$).
However, we may use the ambiguity of fake diamonds to our advantage
and allow ourselves, when necessary, to view a fake diamond as of type $1$ when computing its distance from the previous diamond
and then as of type $0$ when computing its distance from the following diamond.
With this interpretation it remains true that any two given consecutive diamonds in any of the known examples of Nottingham algebras
have distance $q-1$ if suitably interpreted in case both diamonds are fake.
Proving this in great generality is the goal of Sections~\ref{sec:chains} and~\ref{sec:chain_proof}.

Throughout the paper we make extensive use of the {\em generalized Jacobi identity}
\[
[a[bc^n]]=\sum_{i=0}^{n} (-1)^i \binom{n}{i} [ac^{i}bc^{n-i}].
\]
Two special instances which often occur are $[a[bc^q]]=[abc^q]-[ac^qb]$ (which amounts to $(\ad c)^q$ being a derivation), and
$
[a[bc^{q-1}]]=\sum_{i=0}^{q-1}\;[ac^{i}bc^{q-1-i}],
$
due to $\binom{q-1}{i}\equiv (-1)^i\pmod{p}$.
More generally, the binomial coefficients involved in the generalized Jacobi identity can be efficiently evaluated modulo $p$ by means of Lucas' theorem:
if $q$ is a power of $p$ and  $a,b,c,d$ are non-negative
integers with $b,d<q$, then
$
\binom{aq+b}{cq+d}\equiv \binom{a}{c}\binom{b}{d} \pmod p.
$

\section{Diamond distances in Nottingham algebras}\label{sec:chains}

Let $L$ be a Nottingham algebra with second diamond $L_q$ and standard generators $x$ and $y$.
According to~\cite[Lemma~2.2]{CaMa:Nottingham}, as extended in~\cite[Proposition~5.1]{Young:thesis} to cover the case $p=5$,
all homogeneous components from $L_{q+1}$ up to $L_{2q-3}$ are centralized by $y$.
In particular, a third diamond of $L$, if genuine or fake of type $0$, cannot occur in degree lower than $2q-1$.
Combining this information with Proposition~\ref{prop:centr_by_y} below we find that
$L_{2q-1}$ must be a diamond, and with a type.
Hence the distance between the second and the third diamond (properly interpreted when the latter is fake) of a Nottingham algebra is invariably $q-1$.
The main goal of this paper is suitably extending this statement to the distance
between any two consecutive diamonds of a Nottingham algebra.
As we discussed in Section~\ref{sec:types}, to do so we may possibly have to interpret certain fake diamonds in two ways.
We start with the easier part, which is establishing an upper bound on such distance.
In Proposition~\ref{prop:centr_by_y} we approach that task slightly more generally than would be strictly necessary for the goals of this paper.

Before we proceed further we consider which homogeneous components of a Nottingham algebra can be centralized by $y$.
Recall from Section~\ref{sec:types} that no two consecutive components of a Nottingham algebra can be diamonds.
Note also that if $L_m$ is a genuine diamond then $[L_my]\neq 0$.
In fact, if this were not the case then $[wxx]$ and $[wyx]$ would be proportional, where $w$ spans $L_{m-1}$, and then
some linear combination $z$ of $[wx]$ and $[wy]$ would violate the covering property.

In a Nottingham algebra no more than two consecutive components can be not centralized by $y$.
This is Assertion~(a) of
Theorem~\ref{thm:distance_intro},
and is easy to prove.
In fact, if $[L_{m-2}y]=0$ and $[L_{m-1}y]\neq 0$ then $[L_{m-2}x]=L_{m-1}$ by the covering property, whence $\dim(L_{m-1})=1$.
Also, $[L_{m-1}yy]=0$ because $(\ad y)^2=0$, and hence $\dim(L_{m+1})=1$ as well, again by the covering property.
If $\dim(L_m)=1$ then $[L_{m-1}y]=L_m$, and hence $[L_my]=0$.
If $\dim(L_m)=2$ and $w$ spans $L_{m-1}$, then by the covering property $[wyx]$ spans $L_{m+1}$, because $[wyy]=0$.
But then $0=[w[xyy]]=[wxyy]-2[wyxy]+[wyyx]=-2[wyxy]$
implies $[L_{m+1}y]=0$.

The above argument has made essential use of $(\ad y)^2=0$.
Now we use the other crucial relation $(\ad x)^q=0$ to produce an upper bound on the number of consecutive components centralized by $y$,
and further information.

\begin{prop}\label{prop:centr_by_y}
Let $L$ be a Nottingham algebra with second diamond $L_q$ and standard generators $x$ and $y$.
Then the following hold.

\begin{enumerate}
\item[(a)]
No more than $q-1$ homogeneous consecutive components of $L$ can be centralized by $y$.
\item[(b)]
If $y$ centralizes $L_{m},\ldots,L_{m+q-2}$, then
$L_{m}$ is a diamond of type $0$ and $L_{m+q-1}$ is a diamond of type $1$.
\item[(c)]
If $y$ centralizes $L_m,\ldots,L_{m+q-3}$, but not $L_{m-1}$ or $L_{m+q-2}$, then
either $L_m$ is a diamond of type $0$ and $L_{m+q-1}$ is a diamond with a type different from $1$,
or $\dim(L_{m-1})=2$ and $L_{m+q-2}$ is a diamond of type $1$.
\item[(d)]
If $y$ centralizes $L_{m+1},\ldots,L_{m+q-3}$, but not $L_{m+q-2}$, and $\dim(L_m)=2$,
then $L_{m+q-1}$ is a diamond with a type different from $1$.
\end{enumerate}
\end{prop}


\begin{proof}
Suppose
$[L_{m-2}y]=0$ and
$[L_{m-1}y]\neq 0$, for some $m$.
In particular, because $y$ cannot centralize any genuine diamond, $L_{m-2}$ is one-dimensional,
and hence so is $L_{m-1}=[L_{m-2}x]$.
Then $[L_{m-1}yx]=L_{m+1}$ because of the covering property.
Because $(\ad x)^q=0$ we have $[L_{m-1}yx^q]=0$.
Therefor, $[L_{m+j}y]\neq 0$ for some $0< j\le q-1$,
otherwise $L_{m+q}=0$, against infinite-dimensionality of $L$.
Assertion~(a) follows, that no more than $q-1$ consecutive components of $L$ can be centralized by $y$.

Now $[L_{m-1}x]$ may be zero, in which case $L_m$ is a diamond of type $0$, or not.
Also, if $[L_{m-1}x]$ is nonzero then it may equal $[L_{m-1}y]$, or not, in which case $\dim(L_m)=2$.

Suppose $[L_{m-1}x]$ equals $[L_{m-1}y]$, and hence both equal $L_m$.
If $y$ centralizes $L_{m},\ldots,L_{m+q-3}$, as assumed in both Assertions~(b) and (c), then
$L_{m+q-2}=[L_{m-2}x^q]=0$ because $(\ad x)^q=0$.
This contradiction shows that if $y$ centralizes $L_{m},\ldots,L_{m+q-3}$, whence $\dim(L_m)=1$, then
$L_m$ is a diamond of type $0$.

If $y$ centralizes all components from $L_m$ up to $L_{m+q-2}$, then
$L_{m+q-2}=[L_{m-1}yx^{q-2}]$ is one-dimensional, is centralized by $y$, and
$[L_{m+q-2}xx]=[L_{m-1}yx^q]=0$.
Hence $L_{m+q-1}$ is a (fake) diamond of type $1$.
Thus, this establishes Assertion~(b).

Now suppose $y$ centralizes $L_{m},\ldots,L_{m+q-3}$, but not $L_{m+q-2}$.
Then again
$L_{m+q-2}=[L_{m-1}yx^{q-2}]$ is one-dimensional, and it satisfies
$[L_{m+q-2}yy]=0$ because $(\ad y)^2=0$, and also
$[L_{m+q-2}xx]=[L_{m-1}yx^q]=0$.
Hence $L_{m+q-1}$ is a diamond with a type, and that type is not $1$ because $[L_{m+q-2}y]\neq 0$.
This is the first alternative outcome of Assertion~(c),
and we have already shown $[L_{m-1}x]=0$, whence $L_m$ is a diamond of type $0$.

Finally, assume $\dim(L_m)=2$.
Then necessarily $[L_my]\neq 0$,
hence in this case we find no more than $q-2$ consecutive components centralized by $y$.
Also, if $y$ centralizes $L_{m+1},\ldots,L_{m+q-2}$, then
$L_{m+q-2}=[L_{m-1}yx^{q-2}]$ satisfies
$[L_{m+q-2}y]=0$ and
$[L_{m+q-2}xx]=[L_{m-1}yx^q]=0$,
hence $L_{m+q-1}$ is a diamond of type $1$.
This is the second alternative outcome of Assertion~(c), although expressed here with indices increased by one.
However, if $y$ centralizes $L_{m+1},\ldots,L_{m+q-3}$ but not $L_{m+q-2}$, then
again $L_{m+q-1}$ is a diamond with a type, but that type is not $1$.
This proves Assertion~(d).
\end{proof}

The extremal situation described in Assertion~(b) of Proposition~\ref{prop:centr_by_y},
where $y$ centralizes $q-1$ consecutive components,
may also be read as two fake diamonds of the same type (either both $0$ or both $1$) with a degree difference of $q$,
as we illustrated in Section~\ref{sec:types} on Young's algebras $L_1(q)$ and $L_2(q)$.

Note that Assertion~(c) of Proposition~\ref{prop:centr_by_y} includes two possibilities
for having precisely $q-2$ consecutive components centralized by $y$ between a diamond $L_m$ and the next one:
one case is if $L_m$ has type $0$, whence $[L_my]=0$ as well,
and $L_{m+q-1}$ is a diamond not of type $1$;
the other case is if  $\dim(L_m)=2$ and $L_{m+q-1}$ has type $1$, whence $[L_{m+q-2}y]=0$ as well.
If $L_m$ has type $0$ and $L_{m+q-1}$ has type $1$ then we get the situation of Assertion~(b).

Assertion~(d) of Proposition~\ref{prop:centr_by_y}
does not cover all situations where $y$ centralizes a sequence of $q-3$ homogeneous components of $L$ and no more.
In particular, another such situation arises where $L_m$ is a diamond of type $1$ and $L_{m+q-1}$ is a diamond with a type different from $1$,
as in the Nottingham algebras with all diamonds of finite type, see~\cite{CaMa:Nottingham}.

In the above proof of Proposition~\ref{prop:centr_by_y} we had to consider the possibility of $\dim(L_{m-1})=1$ and $[L_{m-1}y]\neq 0$ but $[L_{m-1}x]=L_m$.
That situation actually never occurs in Nottingham algebras, but excluding it requires information on a previous diamond,
and hence can only be established inductively.

We extract some information from Proposition~\ref{prop:centr_by_y}
in the precise form in which we will need it later.

\begin{cor}\label{cor:next_diamond}
Let $L$ be a Nottingham algebra with second diamond $L_q$ and standard generators $x$ and $y$.
Suppose $L_m$ is a diamond of type $\mu\in\F\cup\{\infty\}$,
and suppose $y$ centralizes $L_{m+1},\ldots,L_{m+q-3}$.
If $\mu=1$ assume, in addition, that $y$ does not centralize
$L_{m+q-2}$.
Then $L_{m+q-1}$ is a diamond with a type.
\end{cor}

\begin{proof}
If $\mu\neq 1$ the conclusion follows
from Proposition~\ref{prop:centr_by_y}.
Thus, assume $L_m$ has type $1$, and let $w$ span $L_{m+q-2}$.
By assumption we have $[wy]\neq 0$,
and by the covering property
$[wyx]$ and $[wyy]$ span $L_{m+q}$.
Because $(\ad y)^2=0$ we find that $L_{m+q}$
is spanned by $[wyx]$.
Therefore,
$[wxx]=\alpha [wyx]$ for some scalar $\alpha$.
Consequently,
$[wxxx]=\alpha [wyxx]$,
but $[wxxx]=0$
because $0=[L_m[yx^q]]=[L_{m+q}x]$.
However, because
\[
0=[w[xyy]]=[wxyy]-2[wyxy]+[wyyx]=-2[wyxy],
\]
$[wyxx]$ spans $L_{m+q+1}$, and in particular is nonzero.
Therefore, $\alpha=0$, and hence $[wxx]=0$,
and we conclude that $L_{m+q-1}$ is a diamond
with a type.
\end{proof}


Corollary~\ref{cor:next_diamond} implies, in particular,
that if two consecutive diamonds are both genuine (and with a type) then the difference of their degrees cannot exceed $q-1$.
In fact, we aim to prove that such distance always equals $q-1$.
This statement will extend to when precisely one of two consecutive diamonds is fake,
provided that we make an appropriate choice of its type and degree,
making use of Convention~\ref{convention:fake}.

Proving that requires showing that enough consecutive homogeneous components after each diamond are centralized by $y$.
It is harder to establish, and it will be necessary to proceed inductively, with induction step provided by the following result, which we will prove in Section~\ref{sec:chain_proof}.

\begin{theorem}\label{thm:chain_recap}
Let $L$ be a Nottingham  algebra with second diamond $L_q$ and standard generators $x$ and $y$.
Let $L_{m}$ be a (possibly fake) diamond of $L$,
for some $m\geq 2q-1$,
of type $\mu$.
If $\mu$ equals $-1$ or $0$, assume in addition that $L_{m-q+1}$ is a diamond with a type $\lambda$,
and in case $\mu=0$ assume
$\lambda\neq 0$.
Suppose that $y$ centralizes every homogeneous component from $L_{m-q+2}$ up to $L_{m-2}$.
Then $y$ centralizes each homogeneous component from $L_{m+1}$ up to $L_{m+q-3}$.
\end{theorem}

When $L_m$ has type $-1$, Theorem~\ref{thm:chain_recap}
requires information on the previous diamond,
namely, that it occurs as $L_{m-q+1}$.
One indication of why this setup is more delicate
is that it occurs when $L$ is the graded Lie algebra
associated to the lower central series of the Nottingham group,
or a generalization studied in~\cite{Car:Nottingham},
where complications arise because of Zassenhaus algebras
having one particular central extension.

When $L_m$ has type $0$, Theorem~\ref{thm:chain_recap} excludes
the case where $L_{m-q+1}$ has type $0$ as well.
However, that case is essentially covered,
if needed in applications, by reinterpreting $L_{m-1}$
and $L_{m-q}$ having both type $1$.

Now we combine Theorem~\ref{thm:chain_recap} and
Corollary~\ref{cor:next_diamond} to provide an inductive proof
of Theorem~\ref{thm:distance_intro} of the Introduction,
by proving a more detailed statement first.

\begin{theorem}\label{thm:distance}
Let $L$ be a Nottingham algebra with second diamond $L_q$ and standard generators $x$ and $y$.
Suppose $[L_{m-2}y]=0$ and $[L_{m-1}y]\neq 0$,
for some $m\ge q$.
\begin{enumerate}
\item[(a)]
If $\dim(L_m)=2$, then $y$ centralizes
each homogeneous component from $L_{m+1}$ up to $L_{m+q-3}$, and $L_{m+q-1}$ is a diamond with a type.
\item[(b)]
If $\dim(L_m)=1$,
then either $y$ centralizes
each homogeneous component from $L_m$ up to $L_{m+q-3}$, and $L_{m+q-1}$ is a diamond with a type,
or $y$ centralizes
each homogeneous component from $L_m$ up to $L_{m+q-4}$, and $L_{m+q-2}$ is a diamond with a type.
\end{enumerate}
\end{theorem}

The alternate conclusions in case~(b) of Theorem~\ref{thm:distance}
are due to the fake diamond being interpreted in two ways,
either as $L_m$ of type $0$ or as $L_{m-1}$ of type $1$.

\begin{proof}
When $m=q$ we are in case~(a), and the desired conclusions are known from~\cite{CaMa:Nottingham,Young:thesis},
as we recalled at the beginning of this section.

Proceeding by induction on $m$, we may suppose $m>q$ and assume that
the statement holds for all lower values of $m$.
In particular, this implies that $L_m$ is a diamond with
a type, different from $1$.
Let $L_r$ the latest component before $L_{m-1}$
(and hence before $L_{m-2}$)
which is not centralized by $y$.

Suppose first $\dim(L_r)=2$.
Then $[L_{r-1}y]\neq 0$, and $[L_{r-2}y]=0$,
otherwise $[L_{r-2}y]=L_{r-1}$ because the latter is
one-dimensional, and then $[L_{r-2}yy]\neq 0$,
contradicting $(\ad y)^2=0$.
Thus, by induction $L_r$ has a type, $y$ centralizes
each homogeneous component from $L_{r+1}$ up to $L_{r+q-3}$, and $L_{r+q-1}$ is a diamond with a type.
If the latter type is different from $1$, then
$L_{r+q-1}=L_m$, otherwise $L_{r+q-1}=L_{m-1}$.

In the former case,
according to Theorem~\ref{thm:chain_recap},
$y$ centralizes each homogeneous component
from $L_{m+1}$ up to $L_{m+q-3}$.
According to Corollary~\ref{cor:next_diamond} then
$L_{m+q-1}$ is a diamond with a type.

In the latter case, where $L_{r+q-1}=L_{m-1}$
is a diamond of type $1$, Theorem~\ref{thm:chain_recap}
tells us that $y$ centralizes each homogeneous component
from $L_{m}$ up to $L_{m+q-4}$.
If $y$ does not centralize $L_{m+q-3}$, then
Corollary~\ref{cor:next_diamond} applies to $L_{m-1}$
instead of $L_m$, and yields that
$L_{m+q-2}$ is a diamond with a type.
If $y$ centralizes $L_{m+q-3}$, then
Corollary~\ref{cor:next_diamond} applies to $L_{r+q}=L_m$
viewed as a diamond of type $0$, and yields that
$L_{m+q-1}$ is a diamond with a type.

Suppose now $\dim(L_r)=1$.
Note that $[L_{r-1}y]=0$, because $(\ad y)^2=0$.
Furthermore, $[L_{r+1}y]=0$ by our choice of $r$,
and hence $\dim(L_{r+1})=1$.
Note that by induction $L_r$ is a diamond of type $1$.
Also, according to our induction hypothesis (with $r+1$
playing the role of $m$),
$y$ centralizes each homogeneous component
from $L_{r+1}$ up to at least $L_{r+q-3}$.
Now we distinguish three cases, depending on the earliest
homogeneous component past $L_{r+q-3}$ which is not centralized by
$y$, after noting that $y$ cannot centralize all of
$L_{r+q-2}$, $L_{r+q-1}$ and $L_{r+q}$
because of statement (a) of Proposition~\ref{prop:centr_by_y}.

If $y$ does not centralize $L_{r+q-2}$, then
$L_{r+q-1}=L_m$, by our choice of $r$.
We apply Theorem~\ref{thm:chain_recap} to the diamond $L_m$.
Then the hypothesis of Theorem~\ref{thm:chain_recap}
asking that
$y$ centralizes each homogeneous component
from $L_{m-q+2}$ up to $L_{m-2}$ holds,
and hence so does its conclusion that
$y$ centralizes each homogeneous component
from $L_{m+1}$ up to $L_{m+q-3}$.
(Note that this argument remains valid also in the exceptional
cases of Theorem~\ref{thm:chain_recap} where $L_m$
has type $-1$ or $0$, because $L_{m-q+1}=L_r$ is a diamond
of nonzero type.)
According to Corollary~\ref{cor:next_diamond},
the component $L_{m+q-1}$ is a diamond with a type.

Now suppose $y$ centralizes $L_{r+q-2}$ as well,
but not $L_{r+q-1}$.
Then $L_{r+q}=L_m$ by our choice of $r$.
Thus, $L_{m-q+1}=L_{r+1}$ is a diamond of type $0$
and $y$ centralizes each homogeneous component
from $L_{m-q+2}$ up to $L_{m-2}$.
If $L_m$ does not have type $0$, then
according to Theorem~\ref{thm:chain_recap},
$y$ centralizes each homogeneous component
from $L_{m+1}$ up to $L_{m+q-3}$.
According to Corollary~\ref{cor:next_diamond} then
$L_{m+q-1}$ is a diamond with a type.
If $L_m$ has type $0$, then viewing $L_{m-1}$
as a diamond of type $1$ we may apply
Theorem~\ref{thm:chain_recap}
to it (rather than to $L_m$) and conclude that
$y$ centralizes each homogeneous component
from $L_{m}$ up to $L_{m+q-4}$.
Now if $y$ centralizes $L_{m+q-3}$, then
Corollary~\ref{cor:next_diamond} applies to $L_m$,
a diamond of type $0$, and yields that
$L_{m+q-1}$ is a diamond with a type.
If, instead, $y$ does not centralize $L_{m+q-3}$, then
Corollary~\ref{cor:next_diamond} applies to $L_{m-1}$,
a diamond of type $1$, and yields that
$L_{m+q-2}$ is a diamond with a type.

Finally suppose $y$ centralizes $L_{r+q-2}$
and $L_{r+q-1}$, whence $L_{r+q+1}=L_m$.
Thus, $y$ centralizes each homogeneous component
from $L_{m-q}$ up to $L_{m-2}$.
This is the extremal situation considered in
Assertion (b) of Proposition \ref{prop:centr_by_y},
where $y$ centralizes the highest possible number
of consecutive homogeneous components, and we conclude that
$L_{m-1}$ must be a diamond of type $1$.
Applying Theorem~\ref{thm:chain_recap}
to $L_{m-1}$ instead of $L_m$, we find that
$y$ centralizes each homogeneous component
from $L_m$ up to $L_{m+q-4}$.
If $y$ does not centralize $L_{m+q-3}$, then
Corollary~\ref{cor:next_diamond} applies to $L_{m-1}$
instead of $L_m$, and yields that
$L_{m+q-2}$ is a diamond with a type.
If $y$ centralizes $L_{m+q-3}$, then
Corollary~\ref{cor:next_diamond} applies to $L_m$
viewed as a diamond of type $0$, and yields that
$L_{m+q-1}$ is a diamond with a type.
\end{proof}

\begin{proof}[Proof of Theorem~\ref{thm:distance_intro}]
We proved Assertion~(a) in the main text near the beginning of this section.

As we recalled early in
Section~\ref{sec:types},
the earliest homogeneous component past
$L_1$ which is not centralized by $y$
is $L_{q-1}$, and we also know
that the second diamond $L_q$
has a type,
and in fact equal to $-1$.
Hence in proving Assertions~(b) and~(c)
we may assume $m>q$.

To prove Assertion~(b), note first that if $[L_{m-1}y]\neq 0$
and $[L_{m}y]\neq 0$,
then $L_m$ cannot have dimension one,
otherwise $[L_{m-1}y]=L_m$, and $[L_{m-1}yy]=0$
would yield a contradiction.
Now let $L_{m'}$ the latest homogeneous component of $L$, before $L_{m-1}$, which is not centralized by $y$.
If $\dim(L_{m'})=2$, then case~(a) of Theorem~\ref{thm:distance},
applied to $L_{m'}$ rather than $L_m$,
yields $m'+q-1=m$,
and that $L_m$ is a diamond with a type.
If $\dim(L_{m'})=1$, then case~(b) of Theorem~\ref{thm:distance}
applied to $L_{m'+1}$ yields
$m'+q=m$ or $m'+q-1=m$,
and that $L_m$ is a diamond
with a type in either case.
This proves Assertion~(b).

To prove Assertion~(c), assume $[L_{m-2}y]=0$,
$[L_{m-1}y]\neq 0$
and $[L_{m}y]=0$.
In particular,
here $\dim(L_m)=1$, because $y$ cannot
centralize a two-dimensional component,
as shown earlier in this section.
Let $L_{m'}$ the latest homogeneous component of $L$,
before $L_{m-1}$, which is not centralized by $y$.
If $\dim(L_{m'})=2$, then case~(a) of Theorem~\ref{thm:distance}
applied to $L_{m'}$ shows that either $m'+q-1=m$
and $L_m$ is a diamond with a type different from $1$,
which here necessarily equals $0$, or $m'+q-1=m-1$
and $L_{m-1}$ is a diamond of type $1$.
If $\dim(L_{m'})=1$, then case~(b) of Theorem~\ref{thm:distance}
applied to $L_{m'+1}$ shows that either
$m\in\{m'+q,m'+q-1\}$
and $L_m$ is a diamond with a type different from $1$,
which here necessarily equals $0$, or
$m-1\in\{m'+q,m'+q-1\}$ and $L_{m-1}$ has type $1$.
This proves Assertion~(c).

We now explain how Assertion~(d) also follows from our arguments.

In our proof of Assertion~(b),
in case $\dim(L_{m'})=2$
we showed that
$L_m=L_{m'+q-1}$ is a diamond with a type.
In case $\dim(L_{m'})=1$,
according to Assertion~(c) the
homogeneous component
$L_{m'}$ is a diamond of type $1$,
which may equivalently be
interpreted as $L_{m'+1}$
being a diamond of type $0$.
Then we have proved that
either $L_{m'+q-1}$ or $L_{m'+q}$
is a diamond with a type,
and with the appropriate
interpretation this has degree
$q-1$ higher than the previous diamond.

In our proof of Assertion~(c),
in case $\dim(L_{m'})=2$
we showed that either
$L_m=L_{m'+q-1}$ is a diamond of type $0$,
or $L_{m-1}=L_{m'+q-1}$ is a diamond of type $1$.
Both cases fit the desired conclusion on diamond distance.
In case $\dim(L_{m'})=1$,
it is similarly easy to check that
each of the four possibilities may be suitably
interpreted so that the two consecutive diamonds under
consideration have a difference of $q-1$ in degrees.
\end{proof}

\begin{rem}\label{rem:finite-dim}
The results stated in this section
remain true without our blanket assumption that
$L$ has infinite dimension,
if properly interpreted,
because all our arguments effectively take place in a sufficiently large
finite-dimensional quotient of $L$.
In particular, the conclusions of Theorem~\ref{thm:distance_intro},
on diamonds having a type and on their distances,
remain true for a finite dimensional Nottingham algebra $L$ as long as they do not involve the last
nonzero homogeneous component of $L$, or the preceding one.
\end{rem}

\section{Proof of Theorem~\ref{thm:chain_recap}}\label{sec:chain_proof}

The special case of Theorem~\ref{thm:chain_recap} where $L_{m}$ is a diamond of finite type $\mu$
was, in essence, already known from~\cite[Section~7]{CaMa:Nottingham},
although in a less general setting.
Special considerations were needed there when $\mu=-1,2$.

With Theorem~\ref{thm:chain_recap} we extend the results
of~\cite[Section~7]{CaMa:Nottingham} to the case where $L_{m}$
is a diamond of infinite type.
Furthermore, here we follow a different approach,
where we exploit a natural collection of relations from the start,
while the proofs in~\cite{CaMa:Nottingham}
used a carefully chosen relation at each step,
involving the use of Lucas' theorem to show that a particular
coefficient is nonzero.
Although that approach has its advantages, our present approach
is more systematic and translates the Lie-theoretic deductions
into a simple number-theoretic condition whose implications
appear more transparent, as in Lemma~\ref{lemma:binomial} below.
Our approach includes all cases already covered
in~\cite[Section~7]{CaMa:Nottingham}, and provides
an alternate proof of them in a general setting.

We will need a couple of results on binomial coefficients modulo a prime.
The first result expresses a general symmetry property of
binomial coefficients modulo $p$.

\begin{lemma}\label{lemma:invert}
If $q$ is a power of the prime $p$ we have
\begin{equation*}\label{eq:binomial}
(-1)^{a}\binom{a}{b}\equiv
(-1)^{b}\binom{q-1-b}{q-1-a}
\pmod{p}
\quad\textrm{for $0\le b\le a<q$.}
\end{equation*}
\end{lemma}

\begin{proof}
The special case where $q=p$ follows from Wilson's theorem $(p-1)!\equiv -1\pmod{p}$ after writing
$\binom{a}{b}$ as $a!/\bigl(b!(a-b)!\bigr)$,
and similarly with the other binomial coefficient.
The general case might be less well known, but
two different proofs can be found
in~\cite[Section~4]{Mat:binomial}
and~\cite[Section~2]{Mat:chain_lengths}.
\end{proof}

Our second result on binomial coefficients is tailored to
our application in this paper, as its hypothesis will come up
naturally in the proof of Theorem~\ref{thm:chain_recap}.
That hypothesis asserts that pairs of consecutive entries
of Pascal's triangle's row expressing the coefficients
of $(1+x)^n$, excluding the first and last entries,
stand in a constant proportionality relation
when viewed  modulo a prime $p$.
Two such instances occur, in particular, when
$n$ is a power of $p$, in which case $(1+x)^n=1+x^n$
in $\F_p[x]$,
and when $n+1$ is a power of $p$, in which case
$(1+x)^n=(1+x^{n+1})/(1+x)=\sum_{j=0}^{n+1}(-1)^jx^j$
in $\F_p[x]$.
These examples show that the conclusion of
our Lemma~\ref{lemma:binomial} is best possible.
This result admits a number of variations
(see \cite{Mat:chain_lengths,IMS,Ugo:type_n} for some
in similar contexts as this paper), all aiming
to show that a positive integer $n$ is closely related
to a power of $p$
provided that enough coefficients of $(1+x)^n$
satisfy certain conditions.
The simplest instance is the following well-known and elementary, fact, which we will use repeatedly in the proof of
Lemma~\ref{lemma:binomial}:
if $\binom{n}{j}$ is a multiple of $p$ for $0<j<n$,
where $n$ is a positive integer, then $n$ is a power of $p$.
This can be shown, for example,
by an application of Lucas' theorem.
The prime $p$ in the following result is arbitrary,
we will only need it for $p>3$ in our proof of Theorem~\ref{thm:chain_recap}.

\begin{lemma}\label{lemma:binomial}
Let $q$ be a power of a prime $p$, let $n$ and $a$ be integers with $n>2$,
and suppose
\[
\binom{n}{j-1}+a\binom{n}{j}
\equiv 0\pmod{p}
\qquad\text{for $1<j<n$}.
\]
Then either $n$ is a power of $p$, or $a\equiv 1\pmod{p}$ and $n+1$ is a power of $p$, or $a\equiv -1\pmod{p}$ and $n=3$.
\end{lemma}

\begin{proof}
Consider the polynomial $f(x)=(x+a)(x+1)^n\in\F_p[x]$.
Then our hypothesis expresses the vanishing of
the coefficient of $x^j$ in $f(x)$ in the stated range.
Thus, after working out the remaining coefficients explicitly,
our hypothesis is equivalent to the condition
\begin{equation}\label{eq:pol}
(x+a)(x+1)^n=x^{n+1}+(a+n)x^n+(an+1)x+a
\qquad\text{in $\F_p[x]$,}
\end{equation}
which is somehow more convenient to handle.

If $n$ is a multiple of $p$, then Equation~\eqref{eq:pol}
reads
$(x+a)(x+1)^n=x^{n+1}+ax^n+x+a$ in $\F_p[x]$.
Because the right-hand side factorizes as $(x+a)(x^n+1)$,
we get $(x+1)^n=x^n+1$ in $\F_p[x]$, and so as discussed above
$n$ must be a power of $p$,
which is one of the desired conclusions.
Thus, from now on we assume $n$ is not a multiple of $p$.

Because $-1$ is a root of the left-hand side of
Equation~\eqref{eq:pol}, it must also be a root
of its right-hand side, and hence
$0=(-1)^n(a+n-1)-an-1+a$ in $\F_p$,
This is equivalent to
\begin{equation}\label{eq:an_A}
\bigl(a-(-1)^n\bigr)\bigl(n-1-(-1)^n\bigr)\equiv 0\pmod{p}.
\end{equation}
Moreover, according to Equation~\eqref{eq:pol}
the coefficient of $x^2$ in $(x+a)(x+1)^n\in\F_p[x]$ vanishes,
hence we have $a\binom{n}{2}+n\equiv 0\pmod{p}$, which means
\begin{equation}\label{eq:an_B}
a(n-1)+2\equiv 0\pmod{p}.
\end{equation}
We will now use Equations~\eqref{eq:an_A} and~\eqref{eq:an_B}
jointly to reach the desired conclusions.

Suppose first $n$ is odd.
Then Equation~\eqref{eq:an_A}
yields $a\equiv -1\pmod{p}$.
But then Equation~\eqref{eq:an_B} yields $n\equiv 3\pmod{p}$,
whence $p\neq 3$.
If $p=2$ our hypothesis implies
$(x+1)^{n+1}=x^{n+1}+1$ in $\F_2[x]$,
and hence $n+1$ is a power of $2$.
If $p>3$, then the coefficient of $x^3$ in
$(x+a)(x+1)^n$
equals
$a\binom{n}{3}+\binom{n}{2}\equiv a+3\equiv 2\pmod{p}$.
This contradicts Equation~\eqref{eq:pol} if $n>3$,
and leads to another possible conclusion if $n=3$.

Now suppose $n$ is even.
Then Equation~\eqref{eq:an_A} reads
$(a-1)(n-2)\equiv 0\pmod{p}$.
If $n\not\equiv 2\pmod{p}$, then $a\equiv 1\pmod{p}$.
In turn, Equation~\eqref{eq:an_B} yields $n+1\equiv 0\pmod{p}$.
Then Equation~\eqref{eq:pol} implies
$(x+1)^{n+1}=f(x)=x^{n+1}+1$ in $\F_p[x]$,
and so $n+1$ must be a power of $p$,
which together with $a\equiv 1\pmod{p}$ is the remaining possible conclusion.
If $n\equiv 2\pmod{p}$, whence $p>2$
because we have assumed $n$ prime to $p$, then
the coefficient of $x^3$ in $(x+a)(x+1)^n$ equals
$a\binom{n}{3}+\binom{n}{2}\equiv 1\pmod{p}$,
contradicting Equation~\eqref{eq:pol}.
\end{proof}

After these number-theoretic preliminaries we proceed with the proof proper of Theorem~\ref{thm:chain_recap}.

\begin{prop}\label{prop:chain_genuine}
Let $L$ be a Nottingham algebra with second diamond $L_q$ and standard generators $x$ and $y$.
Let $L_{m}$, with $m\geq 2q-1$, be a diamond of type
$\mu \neq -1,0$.
Suppose $y$ centralizes every homogeneous component from $L_{m-q+3}$ up to $L_{m-2}$.
Then $y$ centralizes each homogeneous component from $L_{m+1}$ up to $L_{m+q-3}$.
\end{prop}

Note that, unlike Theorem~\ref{thm:chain_recap},
Proposition~\ref{prop:chain_genuine}
does not require $y$ to centralizes $L_{m-q+2}$
(but only the following ones up to $L_{m-2}$).
This will provide a bit of leeway, later,
for reducing part of the proof of the case $\mu=0$ in Proposition~\ref{prop:chain_mu=0}
to that of the case $\mu=1$ within
Proposition~\ref{prop:chain_genuine}.

\begin{proof}
We start our proof allowing arbitrary values of $\mu$ with the exception
of $0$, and hence including $-1$.
The argument itself will show its limitations in case $\mu=-1$.

Let the element $v$ span $L_{m-1}$.
Our goal is proving, inductively,
\[
[vxyx^{h-1}y]=0\qquad\text{for $0<h\le q-3$.}
\]
In fact, by induction on $h$ this will imply that
$[vxyx^{h-1}]$ spans $L_{m+h}$,
and will yield the desired conclusion.
To this purpose we will exploit the full range of equations
\[
[yx^jy]=0\qquad\text{for $0\le j\le q-3$.}
\]

Our proof for the induction step will cover also the induction base $h=1$.
Thus, arguing inductively, suppose $1\le h<q-3$, and let $h<j\le q-3$.
Choose an element $u\in L_{m-1-j+h}$ such that $[ux^{j-h}]=v$, as we may,
and note $[uy]=0$.
By hypothesis we have $[ux^iy]=0$ for $0\leq i <j-h$.
Furthermore, because $[vxx]=0$ we have $[ux^{i}]=0$ for $i\geq j-h+2$.
Applying the generalized Jacobi identity, we find
\begin{align*}
0&=
[u[yx^jy]]=[u[yx^j]y]
\\&=
\sum_{i=0}^{j}(-1)^i \binom{j}{i}[ux^{i}yx^{j-i}y]
\\&=
(-1)^{j-h}\binom{j}{h}[vyx^hy]
+(-1)^{j-h+1}\binom{j}{h-1}[vxyx^{h-1}y]
\\&=
(-1)^{j-h}\left(\binom{j}{h}\mu^{-1}-\binom{j+1}{h}\right)[vxyx^{h-1}y],
\end{align*}
where we have used
$\mu([vxy]+[vyx])=[vxy]$
and
$\binom{j}{h-1}+\binom{j}{h}=\binom{j+1}{h}$.

Now the desired conclusion $[vxyx^{h-1}y]=0$ follows, inductively, for $h<q-3$,
provided that we can find $j$ in the range $h<j\le q-3$ such that the coefficient of
$[vxyx^{h-1}y]$ in the above equation is not a multiple of $p$.
That is, we succeed except if
\[
\binom{j+1}{h}\equiv \mu^{-1}\binom{j}{h}\pmod{p}
\qquad\text{for $h<j\le q-3$}.
\]
According to Lemma~\ref{lemma:invert}
this is equivalent to
\[
\binom{q-1-h}{q-2-j}+\mu^{-1}\binom{q-1-h}{q-1-j}\equiv 0\pmod{p}
\qquad\text{for $h<j\le q-3$}.
\]
Now Lemma~\ref{lemma:binomial} with $n=q-1-h$ shows that
either $q-1-h$ is a power of $p$, or $\mu=1$ and $q-h$ is a power of $p$,
or $\mu=-1$ and $h=q-4$.
Note that this covers our induction base $h=1$ because we have assumed $q>5$
in our definition of Nottingham algebras.
Furthermore, the above conclusion points to the obstacle for $\mu=-1$,
which is the reason for excluding that case from the statement of
Proposition \ref{prop:chain_genuine}.

To work around those exceptional values of $h$ in some cases we will need a different calculation,
which is valid for $0\le h\le q-3$:
\begin{align*}
0&=
[v[yx^{h}y]]=[v[yx^{h}]y]-[vy[yx^h]]
\\&=
[vyx^hy]-h[vxyx^{h-1}y]-(-1)^h[vyx^hy]
\\&=
\bigl((1-(-1)^h)(\mu^{-1}-1)-h\bigr)[vxyx^{h-1}y].
\end{align*}
This calculation also takes care of the final step $h=q-3$ of the induction, which was not covered by the previous calculation,
because the coefficient of $[vxyx^{q-4}y]$ equals $3$ and hence is nonzero modulo $p$.

If $q-1-h$ is a power of $p$, necessarily greater than $1$ because $h\le q-3$, then $h$ is odd and congruent to $-1$ modulo $p$.
Hence if $\mu\neq 2$ the above calculation allows us to conclude $[vxyx^{h-1}y]=0$ as desired.
If $\mu=2$ no direct calculation can show the vanishing of
$[vxyx^{h-1}y]=-2[vyx^{h}y]$,
but only that it is a central element, after which
its vanishing follows from the covering property.
In fact, one has $0=[vyx^{h-1}[xy^2]]=[vyx^hy^2]$ and $0=[vx[yx^hy]]=2[vxyx^hy]+[vxyx^{h-1}yx]$.
Therefore $[vyx^hyy]=0=[vyx^{h}yx]$, and so the element $[vyx^hy]$ is central, as claimed.

We finally deal with the case $\mu=1$,
where the above inductive proof that $[vxyx^{h-1}y]=0$
fails to work precisely when $q-h$ is a power of $p$.
We cover those instances of the induction step
with a different calculation.
Thus, suppose $h=q-p^t$, with $1<p^t<q$, and let
$w \in L_{m-p^t}$ such that $[wx^{p^t-1}]=v$.
Because $[v_{1}yx]=-2[v_{1}xy]$ we have
$[w[v_{1}yx]]=-2[w[v_{1}xy]]$,
and we now expand each side.
Up to a nonzero scalar the right-hand side equals
\[
[w[v_{1}xy]]=[w[v_{1}x]y]=[w[yx^{q-1}]y]=[vxyx^{h-1}y],
\]
because $[wy]=0$.
The left-hand side equals
\[
[w[v_{1}yx]]=[w[v_{1}y]x]-[wx[v_{1}y]]=[wv_{1}yx]-[wxv_{1}y]+[wxyv_{1}]=0.
\]
For this calculation we have used $[wxy]=0$,
then that the element $[wv_{1}]$ has degree $m+h-1$, whence it is centralized by $y$ by inductive hypotheses, and finally that
\[
[wxv_{1}]=[wx[yx^{q-2}]]=\binom{q-2}{p^t-1}[vxyx^{h-1}]=0,
\]
where the binomial coefficient is congruent to
$\binom{p^t-2}{p^t-1}$, modulo $p$, and hence
is a multiple of $p$.
We conclude $[vxyx^{h-1}y]=0$ as desired,
which completes the missing steps in our induction.
\end{proof}

We now deal with the remaining cases of Theorem~\ref{thm:chain_recap}, namely,
when $\mu$ equals $-1$ or $0$.


If $L_{m}$ has type $\mu=-1$, then the above proof
of Proposition~\ref{prop:chain_genuine} almost reaches the desired
conclusion, only failing to show that $L_{m+q-4}$
is centralized by $y$.
To extend the conclusion to the case $\mu=-1$,
in Theorem~\ref{thm:chain_recap} we have
added an assumption that $L_{m-q+1}$ is a diamond with a type.

The case where both $L_{m-q+1}$ and $L_{m}$ have type $-1$
occurs when $L$ the graded Lie algebra associated with the lower
central series of the Nottingham group, or a generalization of it
studied in~\cite{Car:Nottingham}.
There it was shown that $[L_{m+q-4},y]$ had to be central
in that case.
In our general setting,
if $L_{m-q+1}$ has finite type different from $-1$, then $[L_{m+q-4},y]=0$ follows from the first calculation in~\cite[Subsection~7.2.1]{CaMa:Nottingham}.
Here, we use an alternate argument to extend the conclusion to the case where the diamond $L_{m-q+1}$ has infinite type.


\begin{prop}\label{prop:chain_mu=-1}
Let $L$ be a Nottingham algebra with second diamond $L_q$ and standard generators $x$ and $y$.
Let $L_{m}$, with $m\geq 2q-1$, be a diamond of type
$\mu=-1$.
Suppose $y$ centralizes every homogeneous component from $L_{m-q+2}$ up to $L_{m-2}$.
Assume further that $L_{m-q+1}$ is a diamond with a type (possibly fake).
Then $y$ centralizes each homogeneous component from $L_{m+1}$ up to $L_{m+q-3}$.
\end{prop}

\begin{proof}
Let $\lambda$ be the type of the diamond $L_{m-q+1}$.
Without loss of generality we may assume $\lambda\neq -1$, the complementary case being covered in~\cite{Car:Nottingham}.
As explained above, the only obstacle in the proof of Proposition~\ref{prop:chain_genuine} when $\mu=-1$ is where
we used the relation $[yx^{q-4}y]=0$.
Thus we only have to show that $[vxyx^{q-5}y]=0$.

Suppose first $\lambda \neq 0$.
We can choose an element  $u\in L_{m-q}$ such that
$v=[uxyx^{q-3}]$. We also have $[uxx]=0=[uyy]$ and $[uyx]=(\lambda^{-1}-1)[uxy]$.
We will use the identities
\begin{align*}
&[u[v_{1}x]]=\lambda^{-1}[vx], &&[ux[v_{1}x]]=0, \\
&[uy[v_{1}x]]=(1-\lambda^{-1})[vxy], &&[uxy[v_{1}x]]=-[vxyx].
\end{align*}
To justify the first one note that $[u[v_{1}x]]=[u[yx^{q-1}]]=[uyx^{q-1}]+[uxyx^{q-2}]=\lambda^{-1}[vx]$. Similarly one can
prove the other ones.
We now show that $[vxyx^{q-5}y]=0$. We use the relation $[v_{1}xyx^{q-5}y]=0$,  expanding
\[
0=[u[v_{1}xyx^{q-5}y]]=[u[v_{1}xyx^{q-5}]y]-[uy[v_{1}xyx^{q-5}]].
\]
The first  term at the right-hand side is
\begin{align*}
&[u[v_{1}xyx^{q-5}]y]=[u[v_{1}xy]x^{q-5}y]+5[ux[v_{1}xy]x^{q-6}y]\\
&=[u[v_{1}x]yx^{q-5}y]-[uy[v_{1}x]x^{q-5}y]+5[ux[v_{1}x]yx^{q-6}y]-5[uxy[v_{1}x]x^{q-6}y]\\
&=(2\lambda^{-1}+4)[vxyx^{q-5}y].
\end{align*}
The second term at the right-hand side is
\begin{align*}
&[uy[v_{1}xyx^{q-5}]]=[uyx^{q-5}[v_{1}x]y]=(\lambda^{-1}-1)[uxyx^{q-6}[v_{1}x]y]\\
&=-(\lambda^{-1}-1)[vxyx^{q-5}y].
\end{align*}
Hence we have
\[
0=[u[v_{1}xyx^{q-5}y]]=(3\lambda^{-1}+3)[vxyx^{q-5}y]
\]
as desired.

If $\lambda=0$ we choose $u \in L_{m-q}$ such that $[uyx^{q-2}]=v$ and $[ux]=0=[uyy]$. The same expansion as above yields
\[
0=3[vxyx^{q-5}],
\]
because
\[
[u[v_{1}xyx^{q-5}]y]=[u[v_{1}xy]x^{q-5}y]=2[vxyx^{q-5}y]
\]
and $[uy[v_{1}xyx^{q-5}]]=-[vxyx^{q-5}y]$.
\end{proof}


We now need to establish a result similar to Proposition~\ref{prop:chain_genuine} for the remaining case where $L_m$ has type $0$.
That can be interpreted as $L_{m-1}$ being a diamond of type $1$.
The conclusions of Proposition~\ref{prop:chain_genuine}
in this shifted setting, that $y$ centralizes each homogeneous
component from $L_{m}$ up to $L_{m+q-4}$,
reach one shorter than we would like.
In order to extend the conclusions to $y$
centralizing $L_{m+q-3}$ as well
we need to make an assumption on a diamond $L_{m-q+1}$.

\begin{prop}\label{prop:chain_mu=0}
Let $L$ be a Nottingham algebra with second diamond $L_q$ and standard generators $x$ and $y$.
Let $L_{m}$, with $m\geq 2q-1$, be a diamond of type
$\mu=0$.
Suppose $y$ centralizes every homogeneous component from $L_{m-q+2}$ up to $L_{m-2}$.
Assume further that $L_{m-q+1}$ is a
(possibly fake) diamond with a nonzero type.
Then $y$ centralizes each homogeneous component from $L_{m+1}$ up to $L_{m+q-3}$.
\end{prop}

\begin{proof}
As explained above, the proof of Proposition~\ref{prop:chain_genuine} yields that
$y$ centralizes each homogeneous component from $L_{m+1}$ up to $L_{m+q-4}$,
hence we only need to prove $[vyx^{q-3}y]=0$.
Choose $u\in L_{m-q}$ such that $v=[uxyx^{q-3}]$.
Since $L_{m-q+1}$ is a diamond of type $\lambda\in \F^{\ast}\cup\{\infty\}$ we also have
$[uxx]=0=[uyy]$ and $[uyx]=(\lambda^{-1}-1)[uxy]$.
We expand
\[
0=[u[v_{1}xyx^{q-4}y]]=[u[v_{1}xyx^{q-4}]y]-[uy[v_{1}xyx^{q-4}]].
\]
The first term of the difference is
\begin{align*}
&[u[v_{1}xyx^{q-4}]y]=[u[v_{1}xy]x^{q-4}y]+4[ux[v_{1}xy]x^{q-5}y]\\
&=[u[v_{1}x]yx^{q-4}y]-[uy[v_{1}x]x^{q-4}y]+4[ux[v_{1}x]yx^{q-5}y]-4[uxy[v_{1}x]x^{q-5}y]\\
&=-(\lambda^{-1}+3)[vyx^{q-3}y].
\end{align*}
The second term of the difference is
\begin{align*}
&[uy[v_{1}xyx^{q-4}]]=\sum_{i=0}^{q-4}(-1)^i\binom{q-4}{i}[uyx^i[v_{1}xy]x^{q-4-i}]\\
&=\sum_{i=0}^{q-4}(-1)^i\binom{q-4}{i}[uyx^i[v_{1}x]yx^{q-4-i}]\\
&=-[uyx^{q-4}[v_{1}x]y]=-(\lambda^{-1}-1)[uxyx^{q-5}[v_{1}x]y]\\
&=-(\lambda^{-1}-1)[vyx^{q-3}y].
\end{align*}
Altogether, we find $4[vyx^{q-3}y]=0$, as desired.
\end{proof}

The results of this section collectively prove
Theorem~\ref{thm:chain_recap},
thus completing also the proofs of
Theorem~\ref{thm:distance} and Theorem~\ref{thm:distance_intro},
which relied on that result.

\bibliography{References}

\end{document}